\documentclass{amsart}
\usepackage{amssymb, amsmath, latexsym, mathabx, dsfont,tikz}
\usepackage{multirow}
\usepackage{enumitem}

\renewcommand{\baselinestretch}{\baselinestretch}
\renewcommand{\baselinestretch}{1.1}
\numberwithin{equation}{section}

\newtheorem{thm}{Theorem}[section]
\newtheorem{lem}[thm]{Lemma}
\newtheorem{cor}[thm]{Corollary}
\newtheorem{prop}[thm]{Proposition}

\theoremstyle{definition}
\newtheorem{conj}[thm]{Conjecture}

\theoremstyle{remark}
\newtheorem{rmk}[thm]{Remark}
\newtheorem{exam}[thm]{Example}

\numberwithin{equation}{section}
\newcommand{\uz}{{u_\mathbb{Z}}}
\newcommand{\ra}{{\,\rightarrow\,}}
\newcommand{\nra}{{\,\nrightarrow\,}}
\newcommand{\gen}{\text{gen}}

\newcommand{\z}{{\mathbb Z}}
\newcommand{\q}{{\mathbb Q}}

\newcommand{\rank}{\text{rank}}

\newcommand{\Mod}[1]{\ (\mathrm{mod}\ #1)}

\newcommand{\df}[1]{\langle #1 \rangle}
\newcommand{\uu}[1][m,n]{\kappa(#1)}
\newcommand{\legendre}[2]{\left( \frac{#1}{#2} \right)}

\newenvironment{newenum}
{\begin{enumerate}[label={\rm(\roman*)}]}
{\end{enumerate}}

\begin{document}
%%%%%%%%%%%%%%%%%%%%%%%%%%%%%%%%%%%%%%%%%%%%%%%%%%%%%%%%%%%%%%%%%%%%%%%%%%%%%%%%%%%%%%%%%%%%%%%%%%%%%%%%%%%%%%%%%%%%
%%%%%%%%%%%%%%%%%%%%%%%%%%%%%%%%%%%%%%%%%%%%%%%%%%%%%%%%%%%%%%%%%%%%%%%%%%%%%%%%%%%%%%%%%%%%%%%%%%%%%%%%%%%%%%%%%%%%
%%%%%%%%%%%%%%%%%%%%%%%%%%%%%%%%%%%%%%%%%%%%%%%%%%%%%%%%%%%%%%%%%%%%%%%%%%%%%%%%%%%%%%%%%%%%%%%%%%%%%%%%%%%%%%%%%%%%
%%%%%%%%%%%%%%%%%%%%%%%%%%%%%%%%%%%%%%%%%%%%%%%%%%%%%%%%%%%%%%%%%%%%%%%%%%%%%%%%%%%%%%%%%%%%%%%%%%%%%%%%%%%%%%%%%%%%
%%%%%%%%%%%%%%%%%%%%%%%%%%%%%%%%%%%%%%%%%%%%%%%%%%%%%%%%%%%%%%%%%%%%%%%%%%%%%%%%%%%%%%%%%%%%%%%%%%%%%%%%%%%%%%%%%%%%
%%%%%%%%%%%%%%%%%%%%%%%%%%%%%%%%%%%%%%%%%%%%%%%%%%%%%%%%%%%%%%%%%%%%%%%%%%%%%%%%%%%%%%%%%%%%%%%%%%%%%%%%%%%%%%%%%%%%
%%%%%%%%%%%%%%%%%%%%%%%%%%%%%%%%%%%%%%%%%%%%%%%%%%%%%%%%%%%%%%%%%%%%%%%%%%%%%%%%%%%%%%%%%%%%%%%%%%%%%%%%%%%%%%%%%%%%
%%%%%%%%%%%%%%%%%%%%%%%%%%%%%%%%%%%%%%%%%%%%%%%%%%%%%%%%%%%%%%%%%%%%%%%%%%%%%%%%%%%%%%%%%%%%%%%%%%%%%%%%%%%%%%%%%%%%
%%%%%%%%%%%%%%%%%%%%%%%%%%%%%%%%%%%%%%%%%%%%%%%%%%%%%%%%%%%%%%%%%%%%%%%%%%%%%%%%%%%%%%%%%%%%%%%%%%%%%%%%%%%%%%%%%%%%
%%%%%%%%%%%%%%%%%%%%%%%%%%%%%%%%%%%%%%%%%%%%%%%%%%%%%%%%%%%%%%%%%%%%%%%%%%%%%%%%%%%%%%%%%%%%%%%%%%%%%%%%%%%%%%%%%%%%
%%%%%%%%%%%%%%%%%%%%%%%%%%%%%%%%%%%%%%%%%%%%%%%%%%%%%%%%%%%%%%%%%%%%%%%%%%%%%%%%%%%%%%%%%%%%%%%%%%%%%%%%%%%%%%%%%%%%
%%%%%%%%%%%%%%%%%%%%%%%%%%%%%%%%%%%%%%%%%%%%%%%%%%%%%%%%%%%%%%%%%%%%%%%%%%%%%%%%%%%%%%%%%%%%%%%%%%%%%%%%%%%%%%%%%%%%

\title[Representations of quadratic forms with same rank]{Representations of finite number of quadratic forms with same rank}

\author{Daejun Kim and Byeong-Kweon Oh}

\address{Department of Mathematical Sciences, Seoul National University, Seoul 08826, Korea}
\email{goodkdj@snu.ac.kr}

\address{Department of Mathematical Sciences and Research Institute of Mathematics, Seoul National University, Seoul 08826, Korea}
\email{bkoh@snu.ac.kr}
\thanks{This work was supported by the National Research Foundation of Korea (NRF-2019R1A2C1086347).}


%%% ----------------------------------------------------------------------

\begin{abstract} Let $m, n$ be positive integers with $m\le n$. Let $\uu$ be the largest integer $k$ such that for any (positive definite and integral) quadratic forms $f_1,\ldots,f_k$ of rank $m$, there exists a quadratic form of  rank $n$ that represents $f_i$ for any $i$ with $1\le i \le k$. In this article, we determine the number $\uu$ for any integer $m$ with  $1\le m\le 8$, except for the cases when $(m,n)=(3,5)$ and $(4,6)$. In the exceptional cases, it will be proved that $1\le \kappa(3,5), \ \kappa(4,6)\le 2$. We also discuss some related topics. 
\end{abstract}

\maketitle

%%%%%%%%%%%%%%%%%%%%%%%%%%%%%%%%%%%%%%%%%%%%%%%%%%%%%%%%%%%%%%%%%%%%%%%%%%%%%%%%%%%%%%%%%%%%%%%%%%%%%%%%%%%%%%%%%%%%
%%%%%%%%%%%%%%%%%%%%%%%%%%%%%%%%%%%%%%%%%%%%%%%%%%%%%%%%%%%%%%%%%%%%%%%%%%%%%%%%%%%%%%%%%%%%%%%%%%%%%%%%%%%%%%%%%%%%
%%%%%%%%%%%%%%%%%%%%%%%%%%%%%%%%%%%%%%%%%%%%%%%%%%%%%%%%%%%%%%%%%%%%%%%%%%%%%%%%%%%%%%%%%%%%%%%%%%%%%%%%%%%%%%%%%%%%
%%%%%%%%%%%%%%%%%%%%%%%%%%%%%%%%%%%%%%%%%%%%%%%%%%%%%%%%%%%%%%%%%%%%%%%%%%%%%%%%%%%%%%%%%%%%%%%%%%%%%%%%%%%%%%%%%%%%
%%%%%%%%%%%%%%%%%%%%%%%%%%%%%%%%%%%%%%%%%%%%%%%%%%%%%%%%%%%%%%%%%%%%%%%%%%%%%%%%%%%%%%%%%%%%%%%%%%%%%%%%%%%%%%%%%%%%
%%%%%%%%%%%%%%%%%%%%%%%%%%%%%%%%%%%%%%%%%%%%%%%%%%%%%%%%%%%%%%%%%%%%%%%%%%%%%%%%%%%%%%%%%%%%%%%%%%%%%%%%%%%%%%%%%%%%
%%%%%%%%%%%%%%%%%%%%%%%%%%%%%%%%%%%%%%%%%%%%%%%%%%%%%%%%%%%%%%%%%%%%%%%%%%%%%%%%%%%%%%%%%%%%%%%%%%%%%%%%%%%%%%%%%%%%
%%%%%%%%%%%%%%%%%%%%%%%%%%%%%%%%%%%%%%%%%%%%%%%%%%%%%%%%%%%%%%%%%%%%%%%%%%%%%%%%%%%%%%%%%%%%%%%%%%%%%%%%%%%%%%%%%%%%
%%%%%%%%%%%%%%%%%%%%%%%%%%%%%%%%%%%%%%%%%%%%%%%%%%%%%%%%%%%%%%%%%%%%%%%%%%%%%%%%%%%%%%%%%%%%%%%%%%%%%%%%%%%%%%%%%%%%
%%%%%%%%%%%%%%%%%%%%%%%%%%%%%%%%%%%%%%%%%%%%%%%%%%%%%%%%%%%%%%%%%%%%%%%%%%%%%%%%%%%%%%%%%%%%%%%%%%%%%%%%%%%%%%%%%%%%
%%%%%%%%%%%%%%%%%%%%%%%%%%%%%%%%%%%%%%%%%%%%%%%%%%%%%%%%%%%%%%%%%%%%%%%%%%%%%%%%%%%%%%%%%%%%%%%%%%%%%%%%%%%%%%%%%%%%
%%%%%%%%%%%%%%%%%%%%%%%%%%%%%%%%%%%%%%%%%%%%%%%%%%%%%%%%%%%%%%%%%%%%%%%%%%%%%%%%%%%%%%%%%%%%%%%%%%%%%%%%%%%%%%%%%%%%
\section{Introduction}
%%%%%%%%%%%%%%%%%%%%%%%%%%%%%%%%%%%%%%%%%%%%%%%%%%%%%%%%%%%%%%%%%%%%%%%%%%%%%%%%%%%%%%%%%%%%%%%%%%%%%%%%%%%%%%%%%%%%
%%%%%%%%%%%%%%%%%%%%%%%%%%%%%%%%%%%%%%%%%%%%%%%%%%%%%%%%%%%%%%%%%%%%%%%%%%%%%%%%%%%%%%%%%%%%%%%%%%%%%%%%%%%%%%%%%%%%
%%%%%%%%%%%%%%%%%%%%%%%%%%%%%%%%%%%%%%%%%%%%%%%%%%%%%%%%%%%%%%%%%%%%%%%%%%%%%%%%%%%%%%%%%%%%%%%%%%%%%%%%%%%%%%%%%%%%
%%%%%%%%%%%%%%%%%%%%%%%%%%%%%%%%%%%%%%%%%%%%%%%%%%%%%%%%%%%%%%%%%%%%%%%%%%%%%%%%%%%%%%%%%%%%%%%%%%%%%%%%%%%%%%%%%%%%
%%%%%%%%%%%%%%%%%%%%%%%%%%%%%%%%%%%%%%%%%%%%%%%%%%%%%%%%%%%%%%%%%%%%%%%%%%%%%%%%%%%%%%%%%%%%%%%%%%%%%%%%%%%%%%%%%%%%
%%%%%%%%%%%%%%%%%%%%%%%%%%%%%%%%%%%%%%%%%%%%%%%%%%%%%%%%%%%%%%%%%%%%%%%%%%%%%%%%%%%%%%%%%%%%%%%%%%%%%%%%%%%%%%%%%%%%
%%%%%%%%%%%%%%%%%%%%%%%%%%%%%%%%%%%%%%%%%%%%%%%%%%%%%%%%%%%%%%%%%%%%%%%%%%%%%%%%%%%%%%%%%%%%%%%%%%%%%%%%%%%%%%%%%%%%
%%%%%%%%%%%%%%%%%%%%%%%%%%%%%%%%%%%%%%%%%%%%%%%%%%%%%%%%%%%%%%%%%%%%%%%%%%%%%%%%%%%%%%%%%%%%%%%%%%%%%%%%%%%%%%%%%%%%
%%%%%%%%%%%%%%%%%%%%%%%%%%%%%%%%%%%%%%%%%%%%%%%%%%%%%%%%%%%%%%%%%%%%%%%%%%%%%%%%%%%%%%%%%%%%%%%%%%%%%%%%%%%%%%%%%%%%
%%%%%%%%%%%%%%%%%%%%%%%%%%%%%%%%%%%%%%%%%%%%%%%%%%%%%%%%%%%%%%%%%%%%%%%%%%%%%%%%%%%%%%%%%%%%%%%%%%%%%%%%%%%%%%%%%%%%
%%%%%%%%%%%%%%%%%%%%%%%%%%%%%%%%%%%%%%%%%%%%%%%%%%%%%%%%%%%%%%%%%%%%%%%%%%%%%%%%%%%%%%%%%%%%%%%%%%%%%%%%%%%%%%%%%%%%
%%%%%%%%%%%%%%%%%%%%%%%%%%%%%%%%%%%%%%%%%%%%%%%%%%%%%%%%%%%%%%%%%%%%%%%%%%%%%%%%%%%%%%%%%%%%%%%%%%%%%%%%%%%%%%%%%%%%

For a positive integer $m$, let $\phi_m(X,Y)$ be the classical modular polynomial (for the definition of this, see \cite{GK}).  
For three positive integers $m_1$,$m_2$, and $m_3$, Gross and Keating \cite{GK} showed that the quotient ring $R(m_1,m_2,m_3)=\z [X,Y]/(\phi_{m_1},\phi_{m_2},\phi_{m_3})$ is finite if and only if there is no positive definite binary quadratic form $Q(x,y)=ax^2+bxy+cy^2$ over $\z$ which represents the three integers $m_1,m_2,m_3$.
Moreover, when $m_1,m_2,m_3$ satisfy this condition, they found an explicit formula for the cardinality of $R(m_1,m_2,m_3)$.
Later, G\"ortz proved  in \cite{Go} that there is no positive definite binary quadratic form $Q(x,y)=ax^2+bxy+cy^2$ over $\z$ which represents  $m_1,m_2,m_3$ if and only if every positive semi-definite half-integral quadratic form with diagonal $(m_1,m_2,m_3)$ is non-degenerate. 

Motivated by the above, we consider the following problem: For positive integers $m$ and $n$, find the largest integer $k$ such that for any positive definite quadratic forms $f_1,f_2,\dots,f_k$ with rank $m$, there is a quadratic form of rank $n$ that represents $f_i$ for any $i$ with $1\le i\le k$. In this article, the largest integer $k$ satisfying the above property will be denoted by $\kappa(m,n)$. 
As a sample, if $m=1$ and $n=2$, which is exactly the above case, then one may easily show that there does not exist a binary quadratic form that represents $1,2$, and $15$. Since there is always a binary quadratic form representing any two positive integers given in advance, we have $\kappa(1,2)=2$. It seems to be very difficult problem to determine the number $\kappa(m,n)$ for arbitrary positive integers $m$ and $n$. The aim of this article is to determine the number $\kappa(m,n)$ for any positive integer $m$ with $m\le 8$ except for the cases when $(m,n)=(3,5)$ and $(4,6)$. In the exceptional cases, we only have 
 $1\le \kappa(3,5), \ \kappa(4,6)\le 2$. We also discuss some related topics.  
 
The subsequent discussion will be conducted in the better adapted geometric language of quadratic spaces and lattices.   A $\z$-lattice $L=\z x_1+\z x_2+\dots+\z x_m$ of rank $m$ is a free $\z$-module equipped with non-degenerate symmetric bilinear form $B$ such that $B(x_i,x_j) \in \q$ for any $i,j$ with $1\le i, j \le m$. The $m\times m$ matrix $(B(x_i,x_j))$ is called the corresponding symmetric matrix to $L$, and we write
$$
L\cong \left(B(x_i,x_j)\right).
$$
The corresponding quadratic map is defined by $Q(x)=B(x,x)$ for any $x \in L$.  If $B(x_i,x_j)=0$ for any $i\neq j$, then we write $L=\langle Q(x_1),\ldots,Q(x_m) \rangle$. 

We say a $\z$-lattice $L$ is {\it positive definite} if $Q(x)>0$ for any non-zero vector $x \in L$, and we say $L$ is {\it integral} if $B(x,y) \in \z$ for any $x,y \in L$. We say $L$ is {\it non-classic integral} if $Q(x) \in \z$ for any $x \in L$. 
Throughout this article, we always assume that a $\z$-lattice is positive definite and integral, unless stated otherwise.  In Section 2, we also consider representations of even integers by an even integral $\z$-lattice, which corresponds to representations of integers by a non-classic integral $\z$-lattice. Here we say a $\z$-lattice $L$ is {\it even} if $Q(x) \in 2\z$ for any $x \in L$. Note that any even $\z$-lattice is integral.  For any positive integer $j$ not greater than $m$, the $j$-th successive minimum of $L$ will be denoted by $m_j(L)$. For the precise definition and various properties on the successive minima, one may consult Chapter 12 of \cite{Cas}.  

For two $\z$-lattices $\ell$ and $L$, we say $\ell$ is represented by $L$ if there is a linear map $\sigma : \ell \to L$ such that 
$$
B(\sigma(x),\sigma(y))=B(x,y) \quad \text{for any $x,y \in \ell$.} 
$$
Such a linear map $\sigma$ is called  {\it an isometry} from $\ell$ to $L$. 
If $\ell$ is (not) represented by $L$, then we will use the notation 
$$
\ell \ra L \quad (\ell \nra L, \text{respectively}).
$$ 

For any positive integer $m$, a $\z$-lattice $L$ is said to be {\em $m$-universal} if $L$ represents all $\z$-lattices of rank $m$.  We define
$$
\uz(m):=\min\{ \rank(L) \, : \, L \text{ is } m\text{-universal}\}.
$$
It is well known that $\uz(m)$ should be greater than or equal to $m+3$. In fact, $\uz(m)=m+3$ for any integer $m$ with $1 \le m \le 5$, and $\uz(m)=13,15,16,28,30$ for $m=6,7,8,9,10$, respectively (see \cite{O}). Furthermore, it is well known that there are only finitely many $m$-universal $\z$-lattices of minimal rank $\uz(m)$ up to isometry, and  for any $m$ with $1\le m \le 8$, the complete lists of candidates of $m$-universal $\z$-lattices with rank $\uz(m)$ can be found in \cite{B}, \cite{C}, \cite{KKO}, \cite{O}, and \cite{R}. 

As stated above, the aim of this article is to determine $\kappa(m,n)$ for any integer $m$ with $1\le m\le 8$. 
If $n \ge \uz(m)$, then there is an $m$-universal  $\z$-lattice. Hence we may naturally define $\uu[m,n]=\infty$ in this case. 
Therefore, we always assume that 
$$
m\le n \le \uz(m)-1.
$$  
For an integer $m$, we define $\mathcal P(m)$ the set of prime divisors of $m$. 

For any integers $a$ and $b$, if $ab^{-1} \in (R^{\times})^2$, then we write $a\sim b$ over $R$, where $R$ is either the $p$-adic integer ring $\z_p$ or $p$-adic field $\q_p$ for some prime $p$. 

Any unexplained notation and terminology can be found in \cite{ki2} or  \cite{om}.

%%%%%%%%%%%%%%%%%%%%%%%%%%%%%%%%%%%%%%%%%%%%%%%%%%%%%%%%%%%%%%%%%%%%%%%%%%%%%%%%%%%%%%%%%%%%%%%%%%%%%%%%%%%%%%%%%%%%
%%%%%%%%%%%%%%%%%%%%%%%%%%%%%%%%%%%%%%%%%%%%%%%%%%%%%%%%%%%%%%%%%%%%%%%%%%%%%%%%%%%%%%%%%%%%%%%%%%%%%%%%%%%%%%%%%%%%
%%%%%%%%%%%%%%%%%%%%%%%%%%%%%%%%%%%%%%%%%%%%%%%%%%%%%%%%%%%%%%%%%%%%%%%%%%%%%%%%%%%%%%%%%%%%%%%%%%%%%%%%%%%%%%%%%%%%
%%%%%%%%%%%%%%%%%%%%%%%%%%%%%%%%%%%%%%%%%%%%%%%%%%%%%%%%%%%%%%%%%%%%%%%%%%%%%%%%%%%%%%%%%%%%%%%%%%%%%%%%%%%%%%%%%%%%
%%%%%%%%%%%%%%%%%%%%%%%%%%%%%%%%%%%%%%%%%%%%%%%%%%%%%%%%%%%%%%%%%%%%%%%%%%%%%%%%%%%%%%%%%%%%%%%%%%%%%%%%%%%%%%%%%%%%
%%%%%%%%%%%%%%%%%%%%%%%%%%%%%%%%%%%%%%%%%%%%%%%%%%%%%%%%%%%%%%%%%%%%%%%%%%%%%%%%%%%%%%%%%%%%%%%%%%%%%%%%%%%%%%%%%%%%
%%%%%%%%%%%%%%%%%%%%%%%%%%%%%%%%%%%%%%%%%%%%%%%%%%%%%%%%%%%%%%%%%%%%%%%%%%%%%%%%%%%%%%%%%%%%%%%%%%%%%%%%%%%%%%%%%%%%
%%%%%%%%%%%%%%%%%%%%%%%%%%%%%%%%%%%%%%%%%%%%%%%%%%%%%%%%%%%%%%%%%%%%%%%%%%%%%%%%%%%%%%%%%%%%%%%%%%%%%%%%%%%%%%%%%%%%
%%%%%%%%%%%%%%%%%%%%%%%%%%%%%%%%%%%%%%%%%%%%%%%%%%%%%%%%%%%%%%%%%%%%%%%%%%%%%%%%%%%%%%%%%%%%%%%%%%%%%%%%%%%%%%%%%%%%
%%%%%%%%%%%%%%%%%%%%%%%%%%%%%%%%%%%%%%%%%%%%%%%%%%%%%%%%%%%%%%%%%%%%%%%%%%%%%%%%%%%%%%%%%%%%%%%%%%%%%%%%%%%%%%%%%%%%
%%%%%%%%%%%%%%%%%%%%%%%%%%%%%%%%%%%%%%%%%%%%%%%%%%%%%%%%%%%%%%%%%%%%%%%%%%%%%%%%%%%%%%%%%%%%%%%%%%%%%%%%%%%%%%%%%%%%
%%%%%%%%%%%%%%%%%%%%%%%%%%%%%%%%%%%%%%%%%%%%%%%%%%%%%%%%%%%%%%%%%%%%%%%%%%%%%%%%%%%%%%%%%%%%%%%%%%%%%%%%%%%%%%%%%%%%
\section{Quadratic forms representing finite number of positive integers}
%%%%%%%%%%%%%%%%%%%%%%%%%%%%%%%%%%%%%%%%%%%%%%%%%%%%%%%%%%%%%%%%%%%%%%%%%%%%%%%%%%%%%%%%%%%%%%%%%%%%%%%%%%%%%%%%%%%%
%%%%%%%%%%%%%%%%%%%%%%%%%%%%%%%%%%%%%%%%%%%%%%%%%%%%%%%%%%%%%%%%%%%%%%%%%%%%%%%%%%%%%%%%%%%%%%%%%%%%%%%%%%%%%%%%%%%%
%%%%%%%%%%%%%%%%%%%%%%%%%%%%%%%%%%%%%%%%%%%%%%%%%%%%%%%%%%%%%%%%%%%%%%%%%%%%%%%%%%%%%%%%%%%%%%%%%%%%%%%%%%%%%%%%%%%%
%%%%%%%%%%%%%%%%%%%%%%%%%%%%%%%%%%%%%%%%%%%%%%%%%%%%%%%%%%%%%%%%%%%%%%%%%%%%%%%%%%%%%%%%%%%%%%%%%%%%%%%%%%%%%%%%%%%%
%%%%%%%%%%%%%%%%%%%%%%%%%%%%%%%%%%%%%%%%%%%%%%%%%%%%%%%%%%%%%%%%%%%%%%%%%%%%%%%%%%%%%%%%%%%%%%%%%%%%%%%%%%%%%%%%%%%%
%%%%%%%%%%%%%%%%%%%%%%%%%%%%%%%%%%%%%%%%%%%%%%%%%%%%%%%%%%%%%%%%%%%%%%%%%%%%%%%%%%%%%%%%%%%%%%%%%%%%%%%%%%%%%%%%%%%%
%%%%%%%%%%%%%%%%%%%%%%%%%%%%%%%%%%%%%%%%%%%%%%%%%%%%%%%%%%%%%%%%%%%%%%%%%%%%%%%%%%%%%%%%%%%%%%%%%%%%%%%%%%%%%%%%%%%%
%%%%%%%%%%%%%%%%%%%%%%%%%%%%%%%%%%%%%%%%%%%%%%%%%%%%%%%%%%%%%%%%%%%%%%%%%%%%%%%%%%%%%%%%%%%%%%%%%%%%%%%%%%%%%%%%%%%%
%%%%%%%%%%%%%%%%%%%%%%%%%%%%%%%%%%%%%%%%%%%%%%%%%%%%%%%%%%%%%%%%%%%%%%%%%%%%%%%%%%%%%%%%%%%%%%%%%%%%%%%%%%%%%%%%%%%%
%%%%%%%%%%%%%%%%%%%%%%%%%%%%%%%%%%%%%%%%%%%%%%%%%%%%%%%%%%%%%%%%%%%%%%%%%%%%%%%%%%%%%%%%%%%%%%%%%%%%%%%%%%%%%%%%%%%%
%%%%%%%%%%%%%%%%%%%%%%%%%%%%%%%%%%%%%%%%%%%%%%%%%%%%%%%%%%%%%%%%%%%%%%%%%%%%%%%%%%%%%%%%%%%%%%%%%%%%%%%%%%%%%%%%%%%%
%%%%%%%%%%%%%%%%%%%%%%%%%%%%%%%%%%%%%%%%%%%%%%%%%%%%%%%%%%%%%%%%%%%%%%%%%%%%%%%%%%%%%%%%%%%%%%%%%%%%%%%%%%%%%%%%%%%%

Recall that for any positive integers $m,n$, we define $\kappa(m,n)$ the largest integer $k$ such that for any $\z$-lattices $\ell_1,\ell_2,\dots,\ell_k$ of rank $m$, there always is a $\z$-lattice of rank $n$ that represents $\ell_i$ for any $i$ with $1\le i\le k$.   
In this section, we determine $\kappa(1,n)$ for any positive integer $n$. Since there is a universal $\z$-lattice of rank $4$,
it suffices to consider the case when $1\le n\le3$.   To begin with, we start proving the following general properties on $\kappa(m,n)$.

\begin{thm}\label{thmgeneral}
Let $m$ and $n$ be positive integers such that $m\le n$, and let $\uu$ be the integer defined above. Then we have the following properties:
\begin{newenum}
     \item $\uu[m,m]=1$ for any positive integer $m$;
	\item $\uu[m+1,n+1] \le \kappa(m,n) \le \kappa(m,n+1)$;
	\item\label{thmgen:3} $\uu[1,2]=2$ and $\uu[m,m+1]=1$ for any $m\ge2$.
\end{newenum}
\end{thm}

\begin{proof}
For the proof of (i), note that there does not exist a $\z$-lattice of rank $m$ that represents both $I_m$ and $I_{m-1}\perp \langle 2\rangle$ simultaneously, where $I_m$ is the $\z$-lattice of rank $m$ whose corresponding symmetric matrix is the identity matrix.    

To prove the first inequality of (ii), let $k=\kappa(m+1,n+1)$. Let $\ell_1,\ldots,\ell_k$ be any $\z$-lattices of rank $m$. From the definition of $\uu[m+1,n+1]$, there exists a $\z$-lattice $L$ of rank $n+1$ which represents $\df{1}\perp\ell_i$ for any $1 \le i \le k$. Since the $\z$-lattice $L$ represents $1$, there is a $\z$-sublattice $L'$ of $L$ such that  $L=\df{1}\perp L'$. Note that the $\z$-lattice $L'$ represents $\ell_i$ for any $i=1,2,\dots,k$. Therefore, we have $\uu[m+1,n+1] \le \uu$. The second inequality is almost trivial.

Now, we prove (iii). Note that for any positive integers $a$ and $b$, the unary $\z$-lattices $\z$-lattices $\df{a}$ and $\df{b}$ are represented by the binary $\z$-lattice $\df{a,b}$. One may easily check that there does not exist a binary $\z$-lattice representing $1,2$, and $15$. Therefore, we have $\kappa(1,2)=2$. To prove the second assertion, suppose that there is a $\z$-lattice, say $L$, of rank $m+1$ that represents both $I_m$ and $I_{m-2}\perp \langle 3,3\rangle$. Then $L \simeq I_{m}\perp \langle a\rangle$ for some positive integer $a$. However, $\langle 3,3\rangle$ is not represented by $I_2 \perp \langle a \rangle$ over $\mathbb Q_2$ for any positive integer $a$. Therefore, we have $\kappa(m,m+1)=1$ for any positive integer $m \ge 2$.      
\end{proof}

In fact, we have the following proposition more general than the first part of \ref{thmgen:3} in Theorem \ref{thmgeneral}.
 
\begin{prop}
For any positive integers $a$ and $b$ which are not contained in the same square class, there are infinitely many positive integers $c$ such that no binary $\z$-lattice represents $a,b$, and $c$ simultaneously.
\end{prop}

\begin{proof} 
Assume that a binary $\z$-lattice $L$ represents both $a$ and $b$, and assume that $Q(x)=a$ and $Q(y)=b$ for some $x,y\in L$. Then two vectors $x$ and $y$ are linearly independent by the hypothesis of the lemma. Hence we have $m_1(L) \le \min (a,b)$ and $m_2(L) \le \max(a,b)$. Therefore, we have $dL\le m_1(L)\cdot m_2(L) \le ab$, and hence there are only finitely many $\z$-lattices of rank $2$, up to isometry,  that represents both $a$ and $b$.
	
Now, let $L_1,\ldots,L_t$ be all such binary $\z$-lattices up to isometry. Let $p$ be any prime such that 
$$
\legendre{-dL_i}{p}=-1 \quad \text{for any } 1\le i \le t,
$$
where $\legendre{ \, \cdot \, }{\cdot}$ denotes the Legendre symbol. Then $p$ is not represented by $(L_i)_p\cong\df{1,-\Delta_p}$ over $\z_p$ for any $i$, where $\Delta_p$ is a non-square unit in $\z_p^\times$. Hence it is not represented by $L_i$ over $\z$ for any $i=1,2,\dots,t$. The lemma follows from this and the fact that there are infinitely many such primes $p$ by the Chinese Remainder Theorem and the Dirichlet's theorem on arithmetic progressions.
\end{proof}

\begin{thm} \label{thm31}
We have $\uu[1,3]=6$.	
\end{thm}
\begin{proof}
Consider the following seven ternary $\z$-lattices:
$$
L(1)=\begin{pmatrix} 2&1&0 \\ 1&2&1 \\ 0&1&3 \end{pmatrix}, \
L(2)=\df{1}\perp \begin{pmatrix} 3&1\\ 1&5 \end{pmatrix}, \
L(3)=\df{1,1,5}, \
$$
$$
L(4)=\df{1}\perp \begin{pmatrix} 2&1\\ 1&2 \end{pmatrix}, \
L(5)=\df{1,2,3}, \
L(6)=\df{1,1,1}, \  \text{and} \ 
L(7)=\df{1,1,2}. \
$$
Note that each of these seven $\z$-lattices has class number $1$, and one may easily check that $L(1),\ldots,L(7)$ represent all positive integers except for the integers of the form 
$$
2^{2s}(8t+1), \ 2^{2s+1}(8t+1), \ 2^{2s}(8t+3), 
$$
$$
\ 2^{2s}(8t+5), \ 2^{2s+1}(8t+5), \ 2^{2s}(8t+7), \ \text{and} \  2^{2s+1}(8t+7),
$$
respectively, where $s$ and $t$ run over all non-negative integers.
Therefore, for any set $S$ of six positive integers, at least one of $L(1),\ldots,L(7)$ represents all integers in the set $S$. This proves that $\uu[1,3]\ge 6$.

On the other hand, we claim that there is no ternary $\z$-lattice that represents $1,2,3,5,10,14$, and $15$ simultaneously. Assume to the contrary that there is a ternary $\z$-lattice, say $L$, which represents all of these seven integers. Since $L$ represents $1$, $L=\df{1}\perp L'$ for some binary $\z$-lattice $L'$. In order for $L$ to represent $2$, the first successive minimum $m_1(L')$ of $L'$ should be less than or equal to $2$. Furthermore, in order for $L$ to represent $3$ or $5$, $L'$ should be isometric to one of the following $11$ binary $\z$-lattices:
\begin{center}
$\df{1,1}, \ \df{1,2}, \ \df{1,3}, \ \df{2,2}, \ \df{2,3}, \ \df{2,4}, \ \df{2,5},$\\ [5pt]
$
\ \begin{pmatrix} 2&1 \\ 1&2 \end{pmatrix}, 
\ \begin{pmatrix} 2&1 \\ 1&3 \end{pmatrix}, 
\ \begin{pmatrix} 2&1 \\ 1&4 \end{pmatrix}, \ \text{and} 
\ \begin{pmatrix} 2&1 \\ 1&5 \end{pmatrix}.
$	
\end{center}

One may easily check that none of $11$ $\z$-lattices of the form $\df{1}\perp L'$ represents integers $1,2,3,5,10,14$, and $15$ simultaneously. Therefore, we have $\uu[1,3]=6$.
\end{proof}

The following proposition says that in many cases, there is a ternary $\z$-lattice that represents all of seven integers given in advance. 
 
\begin{prop}\label{prop31}
Let $A=\{a_1,\ldots,a_7\}$ be a set of positive integers such that there does not exist a ternary $\z$-lattice representing all of integers in the set $A$. Then we have the followings:
\begin{newenum}
\item For any $i\ne j$, $a_i \not \sim a_j$ over $\q_2$, and $a_i \not \sim 6$ over $\q_2$ for any $i=1,2,\dots,7$;
\item  For some $i$ and $j$, $a_i \sim 3$ and $a_j \sim 6$ over $\q_3$. 
\item  For some $i$ and $j$, $a_i \sim 5$ and $a_j \sim 10$ over $\q_5$. 
\item\label{31cond:4}  For some $i$, $a_i$ is odd. Moreover, under the GRH, it belongs to
$$
\{3, 7, 21, 31, 33, 43, 67, 79, 87, 133, 217, 219, 223, 253, 307, 391, 679, 2719\}.
$$
\end{newenum}
\end{prop}

\begin{proof} If the set $A$ does not satisfy the condition (i), then for some $i=1,2,\dots,7$ the ternary $\z$-lattice $L(i)$
 in Theorem \ref{thm31} represents all integers in the set $A$.   

Now, consider the following four ternary $\z$-lattices:
$$
M(1)=\df{1,1,6}, \, M(2)=\df{1,1,3}, \,
M(3)=\df{1}\perp\begin{pmatrix}
2&1\\1&3
\end{pmatrix}, \, 
\text{and} \ M(4)=\df{1,2,5}.
$$
Each of them has class number $1$, and one may easily check that $M(1),M(2),M(3)$, and $M(4)$ represents all positive integers except for the integers of the form 
$$
3^{2s+1}(3t+1), \ 3^{2s+1}(3t+2), \ 5^{2s+1}(5t\pm1), \ 5^{2s+1}(5t\pm2),
$$
respectively. This proves the assertions (ii) and (iii). To prove the assertion \ref{31cond:4}, let us consider the ternary $\z$-lattice $N=\df{1,1,10}$ which corresponds to the Ramanujan's ternary quadratic form. Note that $N$ represents all positive even integers except for those of the form $2^{2s+1}(8t+3)$. Since $A$ does not contain an even integer of the form $2^{2s+1}(8t+3)$ by (i), $N$ represents all even integers in the set $A$. Moreover, under the GRH, all odd integers which are not represented by $N$ are those integers given above (see \cite{OS}). This completes the proof of \ref{31cond:4}.
\end{proof}

In fact, the infinitude of the ring $R(m_1,m_2,m_3)$ defined in the introduction comes from the existence of a non-classic integral binary $\z$-lattice representing $m_1$, $m_2$, and $m_3$. Note that  an integer $a$ is represented by  a non-classic integral $\z$-lattice $L$ if and only if $2a$ is represented by the even $\z$-lattice $L^2$. Here $L^2$ is the $\z$-lattice obtained from $L$ by scaling $2$. Furthermore, for any integral $\z$-lattice $M$, an even integer $2b$ is represented by $M$ if and only if  $2b$ is represented by 
$$
M(e):=\{ x \in L : Q(x) \equiv 0 \Mod 2\},
$$ 
which is an even integral $\z$-lattice.  
Therefore, to deal with a non-classic integral case, we consider the representations of even integers by an integral $\z$-lattice.      

\begin{thm} \label{even-c}
	{\rm (1)} For any subset $A=\{a_1,\ldots,a_7\}$ of even positive integers, there is a ternary $\z$-lattice that represents all integers in the set $A$.\\
	{\rm (2)} For a set $P=\{7,11,13,17,23,29,31,37,39\}$ of prime numbers, let
$$
	N=N_{\alpha}=4\alpha\cdot\prod_{p\in P} p,
$$
	where $\alpha$ is a positive integer satisfying
$$
	\alpha \equiv -\prod_{p\in P} p \Mod{8} \  \text{and} \
	\legendre{\alpha}{p}=
	\begin{cases}
	\displaystyle\prod_{q \in P-\{p\}}\legendre{q}{p} & \text{if } p\in \{11,17\},\\
	-\displaystyle\prod_{q \in P-\{p\}}\legendre{q}{p} & \text{if } p\in P-\{11,17\}.
	\end{cases}
$$
Then there does not exist ternary $\z$-lattice representing
$2,4,6,10,12,14,20$, and  $N$ simultaneously.  
\end{thm}

\begin{proof}
	The assertion (1) follows immediately from \ref{31cond:4} of Proposition \ref{prop31}.
	To prove the assertion (2), assume that there is a ternary $\z$-lattice $L$ representing all of the integers given above. Let $x_1 \in L$ be a vector such that $Q(x_1)=m_1(L)$. Since $L$ represents $2$, we have $Q(x_1)\le 2$.
	
	If $Q(x_1)=1$, then $\z x_1$ does not represent $2$. Hence, $L$ contains a vector $x_2\notin \z x_1$ such that $Q(x_2)=2$. If $Q(x_1)=2$, then $\z x_1$ does not represent $4$, and hence $L$ contains a vector $x_2 \notin \z x_1$ such that $Q(x_2)=4$.  Therefore, $L$ contains a binary $\z$-lattice $\z x_1 + \z x_2$ which is isometric to one of the following binary $\z$-lattices:
	$$
	\begin{pmatrix} 1&0 \\ 0&2 \end{pmatrix}, 
	\ \begin{pmatrix} 1&1 \\ 1&2 \end{pmatrix}, 
	\ \begin{pmatrix} 2&0 \\ 0&4 \end{pmatrix}, 
	\ \begin{pmatrix} 2&1 \\ 1&4 \end{pmatrix},
	\ \text{and} \ \begin{pmatrix} 2&2 \\ 2&4 \end{pmatrix}.
	$$
	Note that each of the above five binary $\z$-lattices does not represent $10,6,10,6$, and $6$, respectively. Hence, $L$ contains a vector $x_3\notin \z x_1+ \z x_2$ such that $Q(x_3)=6$ or $10$. We define a sublattice $L'=\z x_1+\z x_2+\z x_3$ of $L$. For example, if $L$ contains the binary $\z$-lattice $\langle 1,2\rangle$, then
$$
L'=\begin{pmatrix}
1&0&a\\0&2&b\\a&b&	10
\end{pmatrix} \ \ \text{for some integers $a$ and $b$.}
$$ 
Since $L'$ is positive definite, $dL'=20-2a^2-b^2>0$. Therefore, all possible candidates for $(a,b)$ are, up to isometry, 
$$
\begin{array}{rl}
(a,b)=\!\!\!\!&(0,0),(0,1),(0,2),(0,3),(0,4),(1,0),(1,1),(1,2),\\
&(1,3),(1,4),(2,0),(2,1),(2,3),(3,0),(3,1).
\end{array}
$$
%\begin{center}
%$(0,0),(0,1),(0,2),(0,3),(0,4),(1,0),(1,1),(1,2),(1,3),(1,4),$\\ [5pt]
%$(2,0),(2,1),(2,3),(3,0),(3,1)$.
%\end{center}
By considering all possible cases, one may show that there are exactly $52$ candidates for $L'$ up to isometry.
	Note that if the discriminant of $L'$ is square-free, then we have $L=L'$.
	However, if the discriminant of $L'$ is not square-free, $L$ could be a ternary $\z$-lattice properly containing $L'$.
	
	Among those $52$ candidates, there are exactly $34$ lattices which do not represent one of the integers $10,12,14$, and $20$. Furthermore, one may show that any $\z$-lattice properly containing one of those $34$ lattices does not represent the same integer except for the following the only one case:
$$
L'\cong\df{2}\perp \begin{pmatrix} 4&2 \\ 2&8 \end{pmatrix} \quad \text{and}\quad  L\cong \df{2}\perp \begin{pmatrix} 2&1 \\ 1&4 \end{pmatrix}.
$$
In this exceptional case, the $\z$-lattice $L$ does represent $14$, though $L'$ does not represent $14$. However, $L$ does not represent $N$ over $\z_7$ and therefore it does not represent $N$ over $\z$.
	
 Among the remaining  $18(=52-34)$ $\z$-lattices, exactly nine $\z$-lattices are  sublattices of $\df{1,1,1}$.  Those nine $\z$-lattices do not represent the integer $N$, for $\df{1,1,1}$ does not represent any integer of the form $4\cdot (8k+7)$.
Finally, the remaining $9(=18-9)$ $\z$-lattices are one of the followings:
$$
\df{2}\perp \begin{pmatrix} 2&1 \\ 1&4 \end{pmatrix},
\ \df{2}\perp \begin{pmatrix} 2&1 \\ 1&6 \end{pmatrix},
\ \begin{pmatrix} 2&1&0 \\ 1&4&1 \\ 0&1&4 \end{pmatrix},
\ \begin{pmatrix} 2&1&1 \\ 1&4&2 \\ 1&2&6 \end{pmatrix},
\ \df{2}\perp \begin{pmatrix} 4&1 \\ 1&6 \end{pmatrix},
$$
$$
\begin{pmatrix} 2&0&1 \\ 0&4&1 \\ 1&1&8 \end{pmatrix},
\ \df{2}\perp \begin{pmatrix} 4&1 \\ 1&8 \end{pmatrix},
\ \begin{pmatrix} 2&0&1 \\ 0&4&1 \\ 1&1&10 \end{pmatrix}, \ \ \text{and} \ \ 
\ \df{2}\perp \begin{pmatrix} 4&1 \\ 1&10 \end{pmatrix}.
$$
One may directly check that all these $\z$-lattices do not represent the integer $N$ locally. This completes the proof.
\end{proof}

%%%%%%%%%%%%%%%%%%%%%%%%%%%%%%%%%%%%%%%%%%%%%%%%%%%%%%%%%%%%%%%%%%%%%%%%%%%%%%%%%%%%%%%%%%%%%%%%%%%%%%%%%%%%%%%%%%%%
%%%%%%%%%%%%%%%%%%%%%%%%%%%%%%%%%%%%%%%%%%%%%%%%%%%%%%%%%%%%%%%%%%%%%%%%%%%%%%%%%%%%%%%%%%%%%%%%%%%%%%%%%%%%%%%%%%%%
%%%%%%%%%%%%%%%%%%%%%%%%%%%%%%%%%%%%%%%%%%%%%%%%%%%%%%%%%%%%%%%%%%%%%%%%%%%%%%%%%%%%%%%%%%%%%%%%%%%%%%%%%%%%%%%%%%%%
%%%%%%%%%%%%%%%%%%%%%%%%%%%%%%%%%%%%%%%%%%%%%%%%%%%%%%%%%%%%%%%%%%%%%%%%%%%%%%%%%%%%%%%%%%%%%%%%%%%%%%%%%%%%%%%%%%%%
%%%%%%%%%%%%%%%%%%%%%%%%%%%%%%%%%%%%%%%%%%%%%%%%%%%%%%%%%%%%%%%%%%%%%%%%%%%%%%%%%%%%%%%%%%%%%%%%%%%%%%%%%%%%%%%%%%%%
%%%%%%%%%%%%%%%%%%%%%%%%%%%%%%%%%%%%%%%%%%%%%%%%%%%%%%%%%%%%%%%%%%%%%%%%%%%%%%%%%%%%%%%%%%%%%%%%%%%%%%%%%%%%%%%%%%%%
%%%%%%%%%%%%%%%%%%%%%%%%%%%%%%%%%%%%%%%%%%%%%%%%%%%%%%%%%%%%%%%%%%%%%%%%%%%%%%%%%%%%%%%%%%%%%%%%%%%%%%%%%%%%%%%%%%%%
%%%%%%%%%%%%%%%%%%%%%%%%%%%%%%%%%%%%%%%%%%%%%%%%%%%%%%%%%%%%%%%%%%%%%%%%%%%%%%%%%%%%%%%%%%%%%%%%%%%%%%%%%%%%%%%%%%%%
%%%%%%%%%%%%%%%%%%%%%%%%%%%%%%%%%%%%%%%%%%%%%%%%%%%%%%%%%%%%%%%%%%%%%%%%%%%%%%%%%%%%%%%%%%%%%%%%%%%%%%%%%%%%%%%%%%%%
%%%%%%%%%%%%%%%%%%%%%%%%%%%%%%%%%%%%%%%%%%%%%%%%%%%%%%%%%%%%%%%%%%%%%%%%%%%%%%%%%%%%%%%%%%%%%%%%%%%%%%%%%%%%%%%%%%%%
%%%%%%%%%%%%%%%%%%%%%%%%%%%%%%%%%%%%%%%%%%%%%%%%%%%%%%%%%%%%%%%%%%%%%%%%%%%%%%%%%%%%%%%%%%%%%%%%%%%%%%%%%%%%%%%%%%%%
%%%%%%%%%%%%%%%%%%%%%%%%%%%%%%%%%%%%%%%%%%%%%%%%%%%%%%%%%%%%%%%%%%%%%%%%%%%%%%%%%%%%%%%%%%%%%%%%%%%%%%%%%%%%%%%%%%%%
\section{Binary case}
%%%%%%%%%%%%%%%%%%%%%%%%%%%%%%%%%%%%%%%%%%%%%%%%%%%%%%%%%%%%%%%%%%%%%%%%%%%%%%%%%%%%%%%%%%%%%%%%%%%%%%%%%%%%%%%%%%%%
%%%%%%%%%%%%%%%%%%%%%%%%%%%%%%%%%%%%%%%%%%%%%%%%%%%%%%%%%%%%%%%%%%%%%%%%%%%%%%%%%%%%%%%%%%%%%%%%%%%%%%%%%%%%%%%%%%%%
%%%%%%%%%%%%%%%%%%%%%%%%%%%%%%%%%%%%%%%%%%%%%%%%%%%%%%%%%%%%%%%%%%%%%%%%%%%%%%%%%%%%%%%%%%%%%%%%%%%%%%%%%%%%%%%%%%%%
%%%%%%%%%%%%%%%%%%%%%%%%%%%%%%%%%%%%%%%%%%%%%%%%%%%%%%%%%%%%%%%%%%%%%%%%%%%%%%%%%%%%%%%%%%%%%%%%%%%%%%%%%%%%%%%%%%%%
%%%%%%%%%%%%%%%%%%%%%%%%%%%%%%%%%%%%%%%%%%%%%%%%%%%%%%%%%%%%%%%%%%%%%%%%%%%%%%%%%%%%%%%%%%%%%%%%%%%%%%%%%%%%%%%%%%%%
%%%%%%%%%%%%%%%%%%%%%%%%%%%%%%%%%%%%%%%%%%%%%%%%%%%%%%%%%%%%%%%%%%%%%%%%%%%%%%%%%%%%%%%%%%%%%%%%%%%%%%%%%%%%%%%%%%%%
%%%%%%%%%%%%%%%%%%%%%%%%%%%%%%%%%%%%%%%%%%%%%%%%%%%%%%%%%%%%%%%%%%%%%%%%%%%%%%%%%%%%%%%%%%%%%%%%%%%%%%%%%%%%%%%%%%%%
%%%%%%%%%%%%%%%%%%%%%%%%%%%%%%%%%%%%%%%%%%%%%%%%%%%%%%%%%%%%%%%%%%%%%%%%%%%%%%%%%%%%%%%%%%%%%%%%%%%%%%%%%%%%%%%%%%%%
%%%%%%%%%%%%%%%%%%%%%%%%%%%%%%%%%%%%%%%%%%%%%%%%%%%%%%%%%%%%%%%%%%%%%%%%%%%%%%%%%%%%%%%%%%%%%%%%%%%%%%%%%%%%%%%%%%%%
%%%%%%%%%%%%%%%%%%%%%%%%%%%%%%%%%%%%%%%%%%%%%%%%%%%%%%%%%%%%%%%%%%%%%%%%%%%%%%%%%%%%%%%%%%%%%%%%%%%%%%%%%%%%%%%%%%%%
%%%%%%%%%%%%%%%%%%%%%%%%%%%%%%%%%%%%%%%%%%%%%%%%%%%%%%%%%%%%%%%%%%%%%%%%%%%%%%%%%%%%%%%%%%%%%%%%%%%%%%%%%%%%%%%%%%%%
%%%%%%%%%%%%%%%%%%%%%%%%%%%%%%%%%%%%%%%%%%%%%%%%%%%%%%%%%%%%%%%%%%%%%%%%%%%%%%%%%%%%%%%%%%%%%%%%%%%%%%%%%%%%%%%%%%%%
In this section, we consider the binary case. We also introduce the notion that a pair of $\z$-lattices of rank $m$ is ``buried in rank $n$" for some positive integers $m$ and $n$ with $m\le n$, and deal with some related problems. In the previous section, we have proved that $\kappa(2,2)=\kappa(2,3)=1$. To compute the value $\kappa(2,4)$, we need some lemmas.  

\begin{lem} \label{f-bin} For any binary $\z$-lattice $\ell$, there are infinitely many isometry classes of binary $\z$-lattices $\ell'$ such that the number quaternary $\z$-lattices representing both $\ell$ and $\ell'$ is finite up to isometry. 
\end{lem}

\begin{proof} Choose a prime $p$ such that 
$$
p \equiv 3 \Mod 4 \  \ \text{and} \ \ \left(\frac {d\ell}{p}\right)=1.
$$
Note that there are infinitely many primes satisfying these properties. Let $L$ be a quaternary $\z$-lattice representing both $\ell$ and $\ell'(p)=\langle p,p\rangle$. Let $\{ x_i\}$ be a Minkowski reduced basis for $L$ such that $Q(x_i)=m_i(L)$ for each $i=1,2,\dots,4$. Note that such a basis always exists (for this, see \cite{w}). If $\ell$ is not represented by the $3\times 3$ section $L'=\z x_1+\z x_2+\z x_3$ of $L$, then we have $m_4(L) \le m_2(\ell)$ by Lemma 2.1 of \cite{loy}.  Hence the number of possible quaternary $\z$-lattices $L$ is finite  up to isometry. Suppose on the contrary that $\ell$ is represented by $L'$. Then we have $L'_p \cong \df{1,1,\alpha}$ for some $\alpha \in \z_p$. Since $\langle p,p\rangle$ is not represented by $L'$ over $\z_p$, we have $m_4(L) \le m_2(\ell'(p))=p$. This completes the proof.             
\end{proof}

\begin{lem}\label{f-qua}
For any finite number of quaternary $\z$-lattices $L_1,\ldots,L_t$, there are infinitely many isometry classes of binary $\z$-lattices $\ell$ which are not represented by $L_i$ for any $i$ with $1\le i \le t$. 
\end{lem}

\begin{proof}
Without loss of generality, we may assume that there is an integer $j$ with $0\le j\le t$ such that  
$$
dL_i=\begin{cases}
u_i^2 & \text{ for }1\le i \le j,\\
u_i^2\cdot v_i & \text { for } j+1 \le i \le t,
\end{cases}
$$
where $u_i$'s are integers for any $i$ with $1\le i\le t$, and $v_i$'s are square-free integers greater than $1$ for any $i$ with  $j+1\le i\le t$. Choose a prime $p$ such that 
\begin{newenum}
\item $\legendre{v_i}{p}=-1$ for any $i$ with  $j+1\le i \le t$,
\item $(p,2u_1\cdots u_t)=1$,
\item $p\equiv 1 \Mod{4}$.
\end{newenum}
For each $1\le i\le j$, note that if the Hasse symbol $S_2(\q_2L_i)$ is $1$, then $\q_2L_i$ is the anisotropic space. Hence $(L_i)_2$ does not represent any binary isotropic $\z_2$-lattice in this case. On the other hand, if $S_2(\q_2L_i)=-1$, then by the Hilbert Reciprocity Law, there exists an odd prime $q=q_i$ such that $S_q(\q_qL_i)=-1$. Hence $(L_i)_q$ is anisotropic, and therefore $(L_i)_q$ does not represent any binary isotropic $\z_q$-lattice.

For any $i$ with  $j+1\le i \le t$, we have $(L_i)_p \cong \df{1,1,1,\Delta_p}$.
Note that the binary $\z_p$-lattice $\df{p,-p\Delta_p}$ is not represented by $(L_i)_p$.

Now, choose a positive integer $\alpha$ satisfying 
$$
\alpha\equiv 7\Mod{8},\ \legendre{-\alpha}{p}=-1, \text{ and } \legendre{-\alpha}{q_i}=1,
$$
for any $i$ with  $1\le i \le j$ such that $S_2(\q_2L_i)=1$.
Note that there are infinitely many such integers $\alpha$. 
We define a binary $\z$-lattice $\ell=\ell(\alpha)=\df{p,p\alpha}$. Then from the construction of $p$ and $\alpha$, the binary $\z$-lattice $\ell$ is not represented by any of $L_i$'s for any $i$ with $1\le i\le t$, because
$$
\begin{cases}
\begin{array}{llll}
\ell_{2} \!\!\!\!&\cong\df{p,-p}_2\!\!\!\!\! & \nra (L_i)_{2} & \text{if } 1\le i \le j \text{ and }S_2(\q_2L_i)=1,\\
\ell_{q_i}\!\!\!\!&\cong\df{p,-p}_{q_i}\!\!\!\!\!& \nra  (L_i)_{q_i} & \text{if } 1\le i \le j \text{ and } S_2(\q_2L_i)=-1,\\
\ell_p\!\!\!\!&\cong\df{p,-p\Delta_p}_p\!\!\!\!\!& \nra  (L_i)_p & \text{if } j+1\le i \le t.
\end{array}
\end{cases}
$$
This completes the proof.  \end{proof}

\begin{thm}\label{thm3242} For any binary $\z$-lattice $\ell$, there are infinitely many pairs $(\ell_1,\ell_2)$ of isometry classes of binary $\z$-lattices such that there does not exist a quaternary $\z$-lattice representing $\ell$ and $\ell_1, \ell_2$ simultaneously. In particular, we have $\kappa(2,4)=2$.
\end{thm}

\begin{proof}  For any two binary $\z$-lattices $\ell_1$ and $\ell_2$, they are represented by the quaternary $\z$-lattice $\ell_1\perp\ell_2$. Hence we have $\uu[2,4]\ge2$.  Now, the theorem follows directly from Lemmas \ref{f-bin} and \ref{f-qua}.
\end{proof}

\begin{rmk} {\rm As a concrete example of the above theorem, one may easily show that there does not exist a quaternary $\z$-lattice representing $\langle 1,1\rangle, \langle 3,3\rangle$, and $\langle 7,161\rangle$ simultaneously. }  
\end{rmk}

Let $\ell_1$ and $\ell_2$ be $\z$-lattices of rank $m$. We say the pair $(\ell_1,\ell_2)$ of $\z$-lattices is {\it buried in rank $n$} if there is a $\z$-lattice $L$ of rank $n$ representing both $\ell_1$ and $\ell_2$. Similarly, we define the pair $(\ell_1,\ell_2)$ is buried in rank $n$ over $\z_p$ ($\q_p$) if there is a $\z_p$-lattice $L_p$ ($\q_p$-space $V_p$) representing both $\ell_1$ and $\ell_2$ over $\z_p$ ($\q_p$, respectively).
  From the definition, if the pair  $(\ell_1,\ell_2)$ is buried in rank $n$, then it is buried in rank $r$ for any integer $r \ge n$. Clearly, any pair of $\z$-lattices of rank $m$ is buried in rank $2m$. For any binary $\z$-lattice $\ell$, there are infinitely many isometry classes of binary $\z$-lattices $\ell'$ such that $(\ell,\ell')$ is not buried in rank $3$ by Lemma \ref{f-bin}. 

Note that if the pair $(\ell_1,\ell_2)$ of $\z$-lattices is buried in rank $n$ then it should be buried in rank $n$ over $\z_p$ for any prime $p$. 

\begin{prop} \label{buried} Let $\ell_1$ and $\ell_2$ be even $\z$-lattices of rank $m$. Then for any prime $p$, the followings are equivalent.
\begin{newenum}
\item  The pair $(\ell_1,\ell_2)$ is buried in rank $n$ over $\z_p$.
\item   The pair $(\ell_1,\ell_2)$ is buried in rank $n$ over $\q_p$.
\item  As quadratic $\q_p$-spaces, $\q_p\ell_1\simeq \q_p\ell_2$ or $d(\q_p\ell_1)\ne d(\q_p\ell_2)$ and $n=m+1$, or $n\ge m+2$. 
\end{newenum}
\end{prop}

\begin{proof} One may easily show that (i) implies (ii), and (ii) is equivalent to (iii). To show that (ii) implies (i), let $V_p$ be a quadratic space that represents $\q_p\ell_1$ and $\q_p\ell_2$ over $\q_p$. Choose any $2\z_p$-maximal $\z_p$-lattice $L_p$ on $V_p$. From the definition of the maximal lattice, $L_p$ represents both $\ell_1$ and $\ell_2$ over $\z_p$.       
\end{proof}

\begin{rmk} {\rm If $\ell_1$ or $\ell_2$ is not an even $\z$-lattice, then the above proposition does not hold for $p=2$. For example, let $\ell_1=\langle 1,28\rangle$ and $\ell_2= \begin{pmatrix}2&1\\1&2\end{pmatrix}$. Then the pair $(\ell_1,\ell_2)$ is buried in rank $3$ over $\q_2$. In fact, the quadratic $\q_2$-space $\ell_2\perp \langle2\rangle$ of rank $3$ represents $\langle 1,28\rangle$ over $\q_2$. However, there is no $\z_2$-lattice of rank $3$ representing both $\ell_1$ and $\ell_2$.}
\end{rmk}

Let $\ell_1$ and $\ell_2$ be $\z$-lattices of rank $m$. We say the pair $(\ell_1,\ell_2)$ is {\it buried in a genus of rank $n$} if there is a $\z$-lattice $L$ of rank $n$ such that $\ell_1 \,\ra\, L$ and $\ell_2 \,\ra\, L'$ for some $L' \in \gen(L)$.    Note that if the pair $(\ell_1,\ell_2)$ of $\z$-lattices is buried in rank $n$, then it is buried in a genus of rank $n$. The following example shows that the converse does not hold in general.  

\begin{exam}
 Let $\ell_1=\df{1,23}$ and $\ell_2=\df{2,3}$. One may easily show that there is no ternary $\z$-lattice representing both $\ell_1$ and $\ell_2$. However, we see that
$$
\ell_1 \ra L_1=\df{1}\perp \begin{pmatrix} 5&1\\1&23 \end{pmatrix} \quad \text{and} \quad  \ell_2 \ra L_2=\df{2,3,19},
$$
and $L_2\in \gen(L_1)$. Hence the pair $(\ell_1,\ell_2)$ is buried in $\gen(L_1)$ of rank $3$. 
\end{exam}

\begin{prop} \label{local-g} Let $\ell_1, \ell_2$ be even $\z$-lattices of rank $m$. Then the pair $(\ell_1,\ell_2)$ is buried in rank $n$ over $\q_p$ for any prime $p$ if and only if it is buried in a genus of rank $n$.   
\end{prop}

\begin{proof} Note that the ``if" part is trivial. To prove the ``only if" part, we first prove that there is a quadratic $\q$-space $V$ representing both $\q \ell_1$ and $\q \ell_2$. If $n=m$ or $n \ge m+2$, then the existence of such a quadratic $\q$-space follows directly from Proposition \ref{buried}. Assume that $n=m+1$. Let $\alpha$ be a positive integer. Note that by the Local-Global principle, we have
$$
\q\ell_1 \ra  \q \ell_2 \perp \langle \alpha\rangle \ \iff \  (d\ell_1d\ell_2,\alpha)_p=S_p\ell_1\cdot S_p\ell_2\cdot (d\ell_1d\ell_2,d\ell_2)_p \ \ \text{for any prime $p$},
$$ 
where $S_p(\cdot)$ is the Hasse symbol over $\q_p$. One may easily show that there is a subset $\mathcal P_0 \subset \mathcal P(2d\ell_1d\ell_2)$ and a prime $q \not \in \mathcal P(2d\ell_1d\ell_2)$ such that   
$$
\left(d\ell_1d\ell_2,\prod_{r \in \mathcal P_0}r\cdot q\right)_p=S_p\ell_1\cdot S_p\ell_2\cdot (d\ell_1d\ell_2,d\ell_2)_p \ \ \text{for any prime $p \in  \mathcal P(2d\ell_1d\ell_2)$}.  
$$    
Note that if $d\ell_1d\ell_2\sim1$ over $\q_p$, then the above equality holds, for $\q_p\ell_1 \simeq \q_p \ell_2$ from the assumption. For any prime $p \not \in \mathcal P(2d\ell_1d\ell_2) \cup \{q\}$, one may easily check that 
\begin{equation} \label{11}
\q_p\ell_1 \ra  \q_p \ell_2 \perp \langle \prod_{r \in \mathcal P_0}r\cdot q\rangle.
\end{equation}
Since \eqref{11} holds for any prime $p \ne q$, $\q\ell_1$ is represented by $V=\q\ell_2\perp \langle \prod_{r \in \mathcal P_0}r\cdot q\rangle$ by the Hilbert Reciprocity Law and the Local-Global principle.  Now, one may easily show that the pair $(\ell_1,\ell_2)$ is buried in the genus of a $2\z$-maximal lattice on $V$ of rank $n+1$.    
\end{proof}

\begin{cor}  \label{local-g1} Let $\ell_1, \ell_2$ be $\z$-lattices of rank $m$ which are not necessarily even. Then the pair $(\ell_1,\ell_2)$ is buried in rank $n$ over $\z_p$ for any prime $p$ if and only if it is buried in a genus of rank $n$.  
\end{cor}

\begin{proof} Since the proof is quite similar to the above, we only provide the proof of the ``only if" part in the case when $n=m+1$. Assume that $K(2)$ is a $\z_2$-lattice representing both $\ell_1$ and $\ell_2$ over $\z_2$. Then one may suitably choose an integer $\alpha$  in Proposition \ref{local-g}  so that the quadratic $\q$-space $V$ defined there represents $K(2)$ over $\q_2$. Let $K$ be any $2\z$-maximal lattice on $V$. Then both $\ell_1$ and $\ell_2$ are represented by $K$ over $\z_p$ for any prime $p\ne 2$. Now, define a $\z$-lattice $L$ such that
$$
L_p=\begin{cases} K(2) \quad &\text{if $p=2$},\\ 
                  K_p \quad &\text{otherwise}. \end{cases}
$$                                   
Then clearly, the pair $(\ell_1,\ell_2)$ is buried in the $\gen(L)$ of rank $n+1$.   
\end{proof}

We say a $\z$-lattice $L$ is {\it primitive} if there is no integral $\z$-lattice properly containing $L$  on the quadratic $\q$-space $\q L$.  

\begin{thm}\label{32necsuf}
	Let $\ell_1$ and $\ell_2$ be primitive binary $\z$-lattices such that the pair $(\ell_1,\ell_2)$ is not buried in rank $2$. Then the followings are equivalent.
	\begin{newenum}
		\item \label{32cond:1} The pair $(\ell_1,\ell_2)$ is buried in rank $3$.
		\item \label{32cond:2} There exists a positive integer $a$ primitively represented by both $\ell_1$ and $\ell_2$ as well as the following open interval contains an integer:
		$$
		I_{\ell_1,\ell_2,a}:=\left(\frac{b_1b_2-\sqrt{d\ell_1d\ell_2}}{a},\frac{b_1b_2+\sqrt{d\ell_1d\ell_2}}{a}\right),$$
		where $\ell_1=  \setlength\arraycolsep{2pt} \begin{pmatrix} a&b_1\\ b_1&c_1	\end{pmatrix}$ and $\ell_2=\setlength\arraycolsep{2pt}\begin{pmatrix} a&b_2\\ b_2&c_2 \end{pmatrix}$.
	\end{newenum}
\end{thm}
\begin{proof}
First, we prove that \ref{32cond:1} implies \ref{32cond:2}. Let $L$ be a ternary $\z$-lattice containing both $\ell_1$ and $\ell_2$. Let us write 
$$
\ell_1=\z u_1+\z v_1 \text{ and } \ell_2=\z u_2+\z v_2.
$$
Since the vectors $u_1,u_2,v_1$, and $v_2$ are linearly dependent in $L$, there are integers $a,b,c$, and $d$ such that 
$$
x:=au_1+bv_1=cu_2+dv_2  \text { and }  (a,b,c,d)=1.
$$
We claim that $(a,b)=(c,d)=1$. Suppose on the contrary that, without loss of generality, there exists a prime $p$ such that $p\mid (a,b)$ but $p \nmid g:=(c,d)$. Then we have 
$$
\frac{x}{g}=\left(\frac{c}{g}\right)u_2+\left(\frac{d}{g}\right)v_2 = \frac{p}{g}\left[ \left(\frac{a}{p}\right)u_1+\left(\frac{b}{p}\right)v_2\right] \in \ell_2\cap pL.
$$
Hence  $\frac xg$ is a  primitive vector in $\ell_2$, whereas it is not a primitive vector in $L$. This is a contradiction to the assumption that $\ell_2$ is primitive. 
 Therefore, we have $(a,b)=(c,d)=1$, and hence $x$ is a primitive vector in both $\ell_1$ and  $\ell_2$. Moreover, $x$ is also a primitive vector in $L$. Let
$$
\ell_1=\z x + \z x_1 \cong \setlength\arraycolsep{2pt} \begin{pmatrix} a&b_1\\ b_1&c_1	\end{pmatrix} \quad \text{and} \quad
\ell_2=\z x + \z x_2 \cong \setlength\arraycolsep{2pt} \begin{pmatrix} a&b_2\\ b_2&c_2	\end{pmatrix}.
$$
If we consider a ternary $\z$-lattice $L':=\z x + \z x_1 + \z x_2 \subseteq L$, then one may easily verify that $B(x_1,x_2)$ belongs to the open interval $I_{\ell_1,\ell_2,a}$ from the fact that $dL'$ is positive.  

To prove that \ref{32cond:2} implies  \ref{32cond:1}, let $t$ be an integer in the open interval $I_{\ell_1,\ell_2,a}$. Consider a ternary $\z$-lattice
$$
L(t)=\begin{pmatrix} a&b_1&b_2\\ b_1&c_1&t \\ b_2&t&c_2 \end{pmatrix}.
$$
Clearly, $L(t)$ represents both $\ell_1$ and $\ell_2$, and the condition $t\in I_{\ell_1,\ell_2,a}$ implies $dL(t)>0$, that is, $L(t)$ is positive definite. 
\end{proof}

\begin{cor}\label{corsuf32}
Let $\ell_1$ and $\ell_2$ be binary $\z$-lattices. If there is a positive integer $a$ with $a^2\le 4d\ell_1d\ell_2$ that is primitively represented by both $\ell_1$ and $\ell_2$, then the pair $(\ell_1,\ell_2)$ is buried in rank $3$.	
\end{cor}
\begin{proof}
From the hypothesis, we may assume that 
$$
\ell_1=  \setlength\arraycolsep{2pt} \begin{pmatrix} a&b_1\\ b_1&c_1	\end{pmatrix} \quad \text{ and } \quad \ell_2=\setlength\arraycolsep{2pt}\begin{pmatrix} a&b_2\\ b_2&c_2 \end{pmatrix},
$$ 
for some integers $b_1,b_2,c_1,$ and $c_2$.
Let $L(t)$ be the $\z$-lattice defined in the proof of Theorem \ref{32necsuf}. Then $L(t)$ represents both $\ell_1$ and $\ell_2$ for any integer $t$, though it is not necessarily positive definite.
If we choose an integer $t$ such that $|t-b_1b_2/a|\le 1/2$, then
$$
dL(t)=-a(t-b_1b_2)^2+d\ell_1d\ell_2/a \ge -a/4+d\ell_1d\ell_2/a \ge 0.
$$
Therefore, $L(t)$ is a positive semi-definite $\z$-lattice of rank $3$. The corollary follows directly from this. 
\end{proof}

\begin{exam} In general, the converse of Corollary \ref{corsuf32} does not hold. Let us consider the following two binary $\z$-lattices:
$$
\ell_1=\begin{pmatrix} 21&5\\5&64 \end{pmatrix} \quad \text{and} \quad \ell_2=\begin{pmatrix} 24&1\\1&55 \end{pmatrix}.
$$ 
They are in the same genus and $d\ell_1=d\ell_2=1319$. The smallest positive integer $a$ that is primitively represented by both $\ell_1$ and $\ell_2$ is $3080$. 
Note that $a^2>4d\ell_1d\ell_2$ and
$$
\ell_1 \cong \begin{pmatrix} 3080&1321\\1321&567 \end{pmatrix} 
\quad \text{and}\quad 
\ell_2 \cong \begin{pmatrix} 3080&1409\\1409&645 \end{pmatrix}.
$$
Although $\ell_1$ and $\ell_2$ do not satisfy the condition of Corollary \ref{corsuf32}, the ternary $\z$-lattice $L(604)$ defined in the proof of Theorem \ref{32necsuf}, that is,
$$
L(604)=\begin{pmatrix}
3080&1321&1409\\
1321&567&604\\
1409&604&645
\end{pmatrix}
$$
represents both $\ell_1$ and $\ell_2$.
\end{exam}

\begin{conj} Let $\ell_1,\ell_2$ be any binary $\z$-lattices with $d\ell_1=d\ell_2$.  
Moreover, assume that the pair $(\ell_1,\ell_2)$ is buried in a genus of rank $3$.
Then we conjecture that the pair $(\ell_1,\ell_2)$ is buried in rank $3$. We checked this conjecture is true when $d\ell_1=d\ell_2 \le 3000$. 
\end{conj}

%%%%%%%%%%%%%%%%%%%%%%%%%%%%%%%%%%%%%%%%%%%%%%%%%%%%%%%%%%%%%%%%%%%%%%%%%%%%%%%%%%%%%%%%%%%%%%%%%%%%%%%%%%%%%%%%%%%%
%%%%%%%%%%%%%%%%%%%%%%%%%%%%%%%%%%%%%%%%%%%%%%%%%%%%%%%%%%%%%%%%%%%%%%%%%%%%%%%%%%%%%%%%%%%%%%%%%%%%%%%%%%%%%%%%%%%%
%%%%%%%%%%%%%%%%%%%%%%%%%%%%%%%%%%%%%%%%%%%%%%%%%%%%%%%%%%%%%%%%%%%%%%%%%%%%%%%%%%%%%%%%%%%%%%%%%%%%%%%%%%%%%%%%%%%%
%%%%%%%%%%%%%%%%%%%%%%%%%%%%%%%%%%%%%%%%%%%%%%%%%%%%%%%%%%%%%%%%%%%%%%%%%%%%%%%%%%%%%%%%%%%%%%%%%%%%%%%%%%%%%%%%%%%%
%%%%%%%%%%%%%%%%%%%%%%%%%%%%%%%%%%%%%%%%%%%%%%%%%%%%%%%%%%%%%%%%%%%%%%%%%%%%%%%%%%%%%%%%%%%%%%%%%%%%%%%%%%%%%%%%%%%%
%%%%%%%%%%%%%%%%%%%%%%%%%%%%%%%%%%%%%%%%%%%%%%%%%%%%%%%%%%%%%%%%%%%%%%%%%%%%%%%%%%%%%%%%%%%%%%%%%%%%%%%%%%%%%%%%%%%%
%%%%%%%%%%%%%%%%%%%%%%%%%%%%%%%%%%%%%%%%%%%%%%%%%%%%%%%%%%%%%%%%%%%%%%%%%%%%%%%%%%%%%%%%%%%%%%%%%%%%%%%%%%%%%%%%%%%%
%%%%%%%%%%%%%%%%%%%%%%%%%%%%%%%%%%%%%%%%%%%%%%%%%%%%%%%%%%%%%%%%%%%%%%%%%%%%%%%%%%%%%%%%%%%%%%%%%%%%%%%%%%%%%%%%%%%%
%%%%%%%%%%%%%%%%%%%%%%%%%%%%%%%%%%%%%%%%%%%%%%%%%%%%%%%%%%%%%%%%%%%%%%%%%%%%%%%%%%%%%%%%%%%%%%%%%%%%%%%%%%%%%%%%%%%%
%%%%%%%%%%%%%%%%%%%%%%%%%%%%%%%%%%%%%%%%%%%%%%%%%%%%%%%%%%%%%%%%%%%%%%%%%%%%%%%%%%%%%%%%%%%%%%%%%%%%%%%%%%%%%%%%%%%%
%%%%%%%%%%%%%%%%%%%%%%%%%%%%%%%%%%%%%%%%%%%%%%%%%%%%%%%%%%%%%%%%%%%%%%%%%%%%%%%%%%%%%%%%%%%%%%%%%%%%%%%%%%%%%%%%%%%%
%%%%%%%%%%%%%%%%%%%%%%%%%%%%%%%%%%%%%%%%%%%%%%%%%%%%%%%%%%%%%%%%%%%%%%%%%%%%%%%%%%%%%%%%%%%%%%%%%%%%%%%%%%%%%%%%%%%%
\section{Lower rank cases}
%%%%%%%%%%%%%%%%%%%%%%%%%%%%%%%%%%%%%%%%%%%%%%%%%%%%%%%%%%%%%%%%%%%%%%%%%%%%%%%%%%%%%%%%%%%%%%%%%%%%%%%%%%%%%%%%%%%%
%%%%%%%%%%%%%%%%%%%%%%%%%%%%%%%%%%%%%%%%%%%%%%%%%%%%%%%%%%%%%%%%%%%%%%%%%%%%%%%%%%%%%%%%%%%%%%%%%%%%%%%%%%%%%%%%%%%%
%%%%%%%%%%%%%%%%%%%%%%%%%%%%%%%%%%%%%%%%%%%%%%%%%%%%%%%%%%%%%%%%%%%%%%%%%%%%%%%%%%%%%%%%%%%%%%%%%%%%%%%%%%%%%%%%%%%%
%%%%%%%%%%%%%%%%%%%%%%%%%%%%%%%%%%%%%%%%%%%%%%%%%%%%%%%%%%%%%%%%%%%%%%%%%%%%%%%%%%%%%%%%%%%%%%%%%%%%%%%%%%%%%%%%%%%%
%%%%%%%%%%%%%%%%%%%%%%%%%%%%%%%%%%%%%%%%%%%%%%%%%%%%%%%%%%%%%%%%%%%%%%%%%%%%%%%%%%%%%%%%%%%%%%%%%%%%%%%%%%%%%%%%%%%%
%%%%%%%%%%%%%%%%%%%%%%%%%%%%%%%%%%%%%%%%%%%%%%%%%%%%%%%%%%%%%%%%%%%%%%%%%%%%%%%%%%%%%%%%%%%%%%%%%%%%%%%%%%%%%%%%%%%%
%%%%%%%%%%%%%%%%%%%%%%%%%%%%%%%%%%%%%%%%%%%%%%%%%%%%%%%%%%%%%%%%%%%%%%%%%%%%%%%%%%%%%%%%%%%%%%%%%%%%%%%%%%%%%%%%%%%%
%%%%%%%%%%%%%%%%%%%%%%%%%%%%%%%%%%%%%%%%%%%%%%%%%%%%%%%%%%%%%%%%%%%%%%%%%%%%%%%%%%%%%%%%%%%%%%%%%%%%%%%%%%%%%%%%%%%%
%%%%%%%%%%%%%%%%%%%%%%%%%%%%%%%%%%%%%%%%%%%%%%%%%%%%%%%%%%%%%%%%%%%%%%%%%%%%%%%%%%%%%%%%%%%%%%%%%%%%%%%%%%%%%%%%%%%%
%%%%%%%%%%%%%%%%%%%%%%%%%%%%%%%%%%%%%%%%%%%%%%%%%%%%%%%%%%%%%%%%%%%%%%%%%%%%%%%%%%%%%%%%%%%%%%%%%%%%%%%%%%%%%%%%%%%%
%%%%%%%%%%%%%%%%%%%%%%%%%%%%%%%%%%%%%%%%%%%%%%%%%%%%%%%%%%%%%%%%%%%%%%%%%%%%%%%%%%%%%%%%%%%%%%%%%%%%%%%%%%%%%%%%%%%%
%%%%%%%%%%%%%%%%%%%%%%%%%%%%%%%%%%%%%%%%%%%%%%%%%%%%%%%%%%%%%%%%%%%%%%%%%%%%%%%%%%%%%%%%%%%%%%%%%%%%%%%%%%%%%%%%%%%%
In this section, we compute the value $\kappa(m,n)$ for any integer $m$ with $3\le m \le8$.   

\begin{thm}
We have the followings:
\begin{newenum}
\item $\uu=1$ for any $(m,n)$ with $m\le n < \uz(m)-1$ and $3\le m \le 8$;
\item $\uu$ is given as follows when $n=\uz(m)-1$.
\vskip -10pt
\begin{table} [ht]
	\begin{center}
		\begin{tabular}{|c|c|c|c|c|c|c|c|c|}
			\hline
			$(m,n)$ & $(3,5)$ & $(4,6)$ & $(5,7)$ & $(6,12)$ & $(7,14)$ & $(8,15)$ \\
			\hline
			$\uu$ & $1$ {\rm or} $2$ & $1$ {\rm or} $2$ & $1$ & $2$ & $2$ & $1$ \\
			\hline
		\end{tabular}
	\end{center}
\end{table}
\end{newenum}
\end{thm}
\begin{proof}
Note that if $\uu[5,7]=1$, then the results for $3\le m\le 4$ follows immediately by Theorems \ref{thmgeneral} and \ref{thm3242}.
Therefore, we first show that $\uu[5,7]=1$.  In fact, we will show that there is no $\z$-lattice of rank $7$ that represents both
$$
\ell_1:=A_5 \quad \text{and}\quad \ell_2:=I_2\perp K \perp \df{105}, \text{ where } K=\begin{pmatrix} 4&1\\1&4 \end{pmatrix}.
$$
Assume to the contrary that a $\z$-lattice $L$ of rank $7$ represents  the above two quinary $\z$-lattices. Let $L=I_k \perp M$, where $M$ is a $\z$-lattice of rank $7-k$ such that $m_1(M)\ge2$. Since $L$ represents $\ell_2$, we have $k\ge 2$. 
Furthermore, since $L$ represents the indecomposable lattice $\ell_1=A_5$, it should be isometric to either $I_2\perp A_5$ or $I_6\perp\df{t}$ for some positive integer $t$. However, one may easily check that $\ell_2 \nra I_2 \perp A_5$, which implies that $L=I_6\perp\df{t}$ for some positive integer $t$. Now, since $\ell_2$ is represented by $L$, the binary $\z$-lattice $K$ is represented by $I_4\perp\df{t}$.  Since $K$ is not represented by $I_4$,  we have $1\le t \le 3$.  
Under the assumption that $K$ is a sublattice of $I_4\perp\df{t}$,  the orthogonal complement $K^\perp$ of $K$ in $I_4\perp\df{t}$ is isometric to
$$
\begin{pmatrix} 2&1\\1&3 \end{pmatrix} \perp \df{3},\ \df{1,3,10}, \text{ or }
\begin{pmatrix} 2&1\\1&2 \end{pmatrix} \perp \df{15},
$$
according as $t=1,2$, or $3$, respectively. However, $\df{105}$ is not represented by $K^\perp$ in any cases, which contradicts to the assumption that $\ell_2 \ra L$.

Now, assume that $m=6,7$, or $8$. Note that if a $\z$-lattice $L$ represents both $I_m$ and the root lattice $E_m$ of rank $m$, then $L$ should represent $I_m\perp E_m$. This implies that $\uu=1$ for any integer $n$ with $m\le n \le 2m-1$. On the other hand, since 
$$
A_6 \nra I_6 \perp E_6 \qquad \text{and} \qquad A_677\left[2\frac17\right] \nra I_7 \perp E_7,
$$
we have $\uu[6,12]=\uu[7,14]=2$. For the definition of the above $\z$-lattice, see \cite{cs}.  This completes the proof. 
\end{proof}

\end{document}